\documentclass[final,smallcondensed,numbook,dvipsnames]{svjour3}

\usepackage[english]{babel}							
\usepackage[utf8]{inputenc}
\usepackage{amssymb}
\usepackage{graphicx}
\usepackage{amsmath}
\usepackage[parfill]{parskip} 
\usepackage{cancel,mathtools,empheq,latexsym}
\usepackage{subfig,epsfig,tikz,float}		        
\usepackage{booktabs,multicol,multirow,tabularx,array} 
\usepackage{bm}
\usepackage{mathrsfs}
\usepackage{comment}
\usepackage{hyperref}
\hypersetup{
    colorlinks=true,       
    linkcolor=red,          
    citecolor=blue,        
    filecolor=magenta,      
    urlcolor=cyan           
}

\newcommand{\ave}[1]{ \left\{\!\!\left\{ {#1} \right\}\!\!\right\} }
\newcommand{\jump}[1]{ \left[ \! \left[ {#1} \right] \! \right] }

\newcommand{\bs}[1]{\boldsymbol{#1}}

\title{A symmetric boundary integral formulation\\ for time--domain acoustic-elastic scattering}
\titlerunning{Symmetric BIE for time--domain acoustic/elastic scattering}
\dedication{Dedicated to Prof. George C. Hsiao, esteemed collaborator and mentor on the occasion of his 90th. birthday}

\author{Tonatiuh S\'anchez-Vizuet
    \thanks{Partially funded by NSF grant NSF-DMS-2137305. }
    }
\institute{Tonatiuh S\'anchez-Vizuet
    \at Department of Mathematics, The University of Arizona, USA
        \email{tonatiuh@arizona.edu}
    }

\begin{document}

\maketitle

\begin{abstract}
A symmetric boundary integral formulation for the transient scattering of acoustic waves off homogeneous and isotropic elastic obstacles is analyzed. Both the acoustic scattered field and the elastodynamic excited field are represented through a direct integral representation, resulting in a coupled system of interior/exterior integral equations that is symmetrized through the introduction of an auxiliary mortar variable. The analysis of each system and of its Galerkin discretization is done through the passage to the Laplace domain, which allows for the use of convolution quadrature for time discretization. Since the operators of the acoustic and elastic Calder\'on calculus appear independently of each other, the formulation is well suited for non-intrusive numerical implementations (i.e. existing codes for acoustic and elastic problems can be used without any modification).
\end{abstract}

\section{The acoustic-elastic scattering problem}
%
\textbf{Problem setting in the time domain.} We will revisit the problem considered in \cite{HsSaSa2016}, namely that of an acoustic linear wave travelling in an unbounded homogeneous medium and scattering off a bounded elastic homogeneous and isotropic obstacle. The scatterer will be modelled as an open bounded domain $\Omega_-$, not necessarily simply connected (or even connected) with Lipschitz boundary $\Gamma$ and exterior unit normal vector $\bs \nu$. The unbounded acoustic domain, complementary to the closure of the elastic domain, will be denoted as $\Omega_+:=\mathbb R^d\setminus\overline\Omega_-$; in this study $d$ will be either $2$ or $3$.

The fluid, with constant density $\rho_f$, will be assumed to be homogeneous, isotropic and irrotational. The velocity field in the fluid, $\bs v^{tot}$, will be expressed as the superposition of the known incident perturbation $\bs v^{inc}$ and the unknown scattered field $\bs v$ so that $\bs v^{tot} = \bs v + \bs v^{inc}$. Due to the homogeneity and isotropy, perturbations will propagate with constant speed $c$ while, due to the assumption of irrotational flow, the velocity vector field $\bs v$ can be expressed in terms of a scalar potential as $\bs v = \nabla v$, \cite{ChMa1993}. The scalar \textit{velocity potential} $v$ will be our variable of interest and we will require it to be causal so that $v(t)=0$ for all $t\leq 0$, this requirement is akin to the assumption that, at the initial time, the incident wave has not yet incided on the obstacle.

In the elastic medium, perturbations from the reference configuration will be described in terms of the elastic displacement field $\bs u$ and will propagate with speeds depending on the density $\rho_\Sigma$, which will be assumed to be constant, and on the elastic properties of the material encoded in the fourth-order elastic stiffness tensor ${\bf C}$. The case of variable coefficients was considered in \cite{HsSaSa2016,HsSaSaWe:2018,HsSa2021} in conjunction with additional physical properties of the scatterer.

In the linear regime that we will be interested in, the stifness tensor relates the elastic stress tensor $\bs \sigma$ and the linearized strain tensor $\bs \varepsilon := \tfrac{1}{2}\left(\nabla\bs u + \nabla\bs u^\top\right)$ through the relation $\bs \sigma = \bf C\bs\varepsilon$. Whenever there could be any ambiguity regarding the displacement function associated to the stress and the strain, we will dispell it by indicating it explicitly in the form $\bs\varepsilon(\bs u)$ and $\bs\sigma(\bs u)$. 

In terms of its components, the action of $\bf C$ on a matrix $\bf M$ is given by
\[
\left(\bf C\bf M\right)_{ij} = \sum_{k,l} {\bf C}_{ijkl}{\bf M}_{kl}.
\]
We will require the components of $\bf C$ to be constants and to satisfy the symmetry relations \cite{MaHu1983}:
\[
{\bf C}_{ijkl} = {\bf C}_{klij} = {\bf C}_{jikl},
\]
and to be such that the stiffness tensor is positive in the following sense: for any symmetric matrix $\bf M$, there is a positive constant $c_0$ such that
\[
{\bf C\, M}:{\bf M}\geq c_0 {\bf M}:{\bf M}. 
\]
Above, the colon ``\,:\," denotes the Frobenius inner product of matrices ${\bf A}:{\bf B} :=\sum{i,j} {\bf A}_{ij}{\bf B}_{ij}$. 

In terms of the displacement field $\bs u$ and the velocity potential $v$, the equations governing the interaction are:
\begin{subequations}
\label{eq:TDequations}
\begin{alignat}{6}
\label{eq:TDequationsA}
-\nabla\cdot\bs\sigma + \rho_\Sigma\ddot{\bs u} =\,& \bs 0 &\qquad& \text{ in } \Omega_-\times[0,\infty),\\
\label{eq:TDequationsB}
-\Delta v + c^{-2}\ddot v =\,& 0 & \qquad & \text{ in } \Omega_+\times[0,\infty),\\
\label{eq:TDequationsC}
\bs\sigma\bs\nu + \rho_f\dot v \bs\nu =\,& -\rho_f\dot v^{inc}\bs\nu & \qquad & \text{ on } \Gamma \times [0,\infty),\\
\label{eq:TDequationsD}
\dot{\bs u}\cdot\bs\nu + \partial_\nu v =\,& -\partial_\nu v^{inc} & \qquad & \text{ on } \Gamma \times [0,\infty),\\
\label{eq:TDequationsE}
\bs u = \dot{\bs u} =\,& \bs 0 & \qquad & \text{ in } \Omega_-\times\{ 0 \},\\
\label{eq:TDequationsF}
 v = \dot v =\,& 0 & \qquad & \text{ in } \Omega_+\times\{ 0 \}.
\end{alignat}
\end{subequations}
Above, the dot is used to represent differentiation with respect to time. The first two equations are respectively the Navier--Lam\'e elastic wave equation and the acoustic wave equation, while the transmission condition \eqref{eq:TDequationsC} is the result of the equilibrium between the elastic and fluid normal forces along the contact surface, and condition \eqref{eq:TDequationsD} expresses the continuity of the normal velocity fields (i.e. both the elastic displacement and the acoustic velocity field have the same normal component at the interface so that no vaccum is created). Due to the choice of problem unknowns, the system above  is known as a ``velocity-displacement" formulation. An alternative formulation in terms of acoustic pressure and elastic displacement---known as pressure-displacement formulation---is described in \cite{Ihlenburg1998}.

\textbf{Sobolev space preliminaries and notation.} We will make use of standard results and notation from the theory of Sobolev spaces. We introduce the basic terminology here and refer the reader to classic references, such as \cite{AdFo2003} for further details. Let us start by defining the Sobolev spaces
\[
{\bf H} : = \left\{ \bs u\in {\bf H}^1(\mathbb R^d\setminus\Gamma): \nabla\cdot\bs\sigma\in{\bf L}^2(\mathbb R^d\setminus\Gamma)\right\} \qquad \text{ and } \qquad {\rm H}:= \left\{ v \in {\rm H}^1(\mathbb R^d\setminus\Gamma): \Delta v \in{\rm L}^2(\mathbb R^d\setminus\Gamma)\right\}.
\]
The reader may have noticed that the spaces above do not consider time derivatives of the acoustic and elastic variables. The reason for this is that the system will be transformed into the Laplace domain, where time differentiation is represented by a simple multiplication. This greatly simplifies the analysis and will enable us to sidestep the analysis of a hyperbolic equation in favor of a, more familar, elliptic equation.

For functions in these spaces, the trace operator $\gamma^{\pm}$ that restricts the values to the boundary $\Gamma$ is well defined. The superscript ``$+$" or ``$-$" on $\gamma$ will depend on whether the trace is being taken from values on $\Omega_+$ or $\Omega_-$. The trace spaces are denoted as ${\bf H}^{1/2}(\Gamma)$ for the elastic variables and ${\rm H}^{1/2}(\Gamma)$ for the acoustic ones. The duals of these spaces, denoted respectively as ${\bf H}^{-1/2}(\Gamma)$ and ${\rm H}^{-1/2}(\Gamma)$, coincide with the spaces of conormal derivatives defined respectively through the integration by parts formulas
\begin{align*}
\mp\left\langle\bs\sigma(\bs u)\bs\nu^\pm,\gamma^\pm \bs w\right\rangle_\Gamma :=\,& (\nabla\cdot\bs\sigma(\bs u),\bs w)_{\Omega_\pm} + (\bs\sigma(\bs u),\bs\varepsilon(\bs w))_{\Omega_\pm} \,,\\
\mp\left\langle\partial_\nu^{\pm}v,\gamma^\pm w\right\rangle_\Gamma :=\,& (\Delta v,w)_{\Omega_\pm} + (\nabla v,\nabla w)_{\Omega_\pm}\,.
\end{align*}
for functions $\bs u,\bs w \in {\bf H}$ and $v, w \in {\rm H}$. We will also make use of the jump and average traces
\[
\jump{\gamma v} : = \gamma^-v - \gamma^+v \qquad \text{ and } \qquad \ave{\gamma v} : = \tfrac{1}{2}\left(\gamma^-v + \gamma^+v\right),
\]
which are defined in an analogous manner for vector-valued functions and co-normal derivatives.  The action of the normal vector ${\bs \nu}$ as a mapping from the Dirichlet trace space into the Neumann trace space will be captured by the operators
\[
\begin{array}{rcl c rcl}
{\bf N} : {\bf H}^{1/2}(\Gamma) & \; \longrightarrow \; & {\rm H}^{-1/2}(\Gamma) & \qquad \text{ and } \qquad \qquad & {\bf N}^\prime : {\rm H}^{1/2}(\Gamma) & \; \longrightarrow \; & {\bf H}^{-1/2}(\Gamma)\\
 {\bs \varphi} & \longmapsto & {\bs \varphi}\cdot{\bs \nu} & & \varphi & \longmapsto & \varphi\,{\bs\nu}.
\end{array} 
\]

\textbf{Passage to the Laplace domain.} The system \eqref{eq:TDequations} can be riguourously understood in terms of causal distributions taking values in the Sobolev space $\bf H\times\rm H$. When considered as a Sobolev--space valued dynamical system, \eqref{eq:TDequations} can be analyzed directly in the time domain as was done in \cite[Section 7.1]{BrSaSa2018} extending the techniques developed in \cite{HaQiSaSa2015}. However, we will instead transform it into the Laplace domain and, following the seminal work of Laliena and Sayas \cite{LaSa2009}, recast it in terms of boundary integral equations (BIEs), analyze the resulting Laplace domain system, and translate the Laplace--domain estimates into time domain results at the end. This approach is advantageous both from the analytical and computational point of view. In the Laplace domain we can make use of the rich theory of elliptic boundary value problems to establish the well posedness of the problem and regularity requirements for problem data. From the computational point of view, the Laplace--domain integral equations are tamer than their time--domain counterparts. Moreover, their semidiscretization in space can be treated with efficient frequency--domain tools and paired with Lubich's convolution quadrature \cite{Lubich1988,Lubich2004,LuSc1992} to yield stable and accurate fully discrete systems. 

In the interest of brevity, we will not delve in the intricacies of Banach--space valued distributions and their Laplace transforms. Let us just state that the process of Laplace transforming the distributional interpretation of the system \eqref{eq:TDequations} is well defined, and that the Laplace transform of distributions of this kind behaves---at least formally---just like that of functions, so that all the forthcoming manipulations can be fully justified riguourously---the details can be found in \cite{Sayas2016}. Moreover, since in what follows we will almost exclusively be working in the Laplace domain, we will abuse notation and work under the convention that \textit{the Laplace transform of a time--domain function will be represented with the same symbol as its time--domain counterpart}. Whether an argument is being made in the time domain or in the Laplace domain will be clear from the context---for instance by the presence of the Laplace parameter $s$. Hence, the reader is hereby forewarned that, \textbf{\textit{until further notice, everything that follows should be understood as taking place in the Laplace domain}}.

From now on, and for all $s\in\mathbb C_+ :=\left\{s\in\mathbb C: {\rm Re}\, s >0 \right\}$, we will be considering then the problem of finding $(\bs u, v)\in {\bf H}^1(\Omega_-)\times \rm H^1(\Omega_-)$ such that
\begin{subequations}
\label{eq:LDequations}
\begin{alignat}{6}
\label{eq:LDequationsA}
-\nabla\cdot\bs\sigma + \rho_\Sigma s^2{\bs u} =\,& \bs 0 &\qquad& \text{ in } \Omega_-,\\
\label{eq:LDequationsB}
-\Delta v + (s/c)^{2} v =\,& 0 & \qquad & \text{ in } \Omega_+,\\
\label{eq:LDequationsC}
\bs\sigma\bs\nu^- + \rho_f s \gamma^+v \bs\nu =\,& -\rho_f s \phi_0\bs\nu & \qquad & \text{ on } \Gamma ,\\
\label{eq:LDequationsD}
s{\bs u}\cdot\bs\nu^- + \partial_\nu^+ v =\,& -\lambda_0 & \qquad & \text{ on } \Gamma,
\end{alignat}
\end{subequations}
where $\phi_0:=\gamma^+v^{inc}\in{\rm H}^{1/2}(\Gamma)$ and $\lambda_0:=\partial_{\nu}^+v^{inc}\in{\rm H}^{-1/2}(\Gamma)$. An important observation is that, since the elastic and acoustic waves are suported respectively on $\Omega_-$ and $\Omega_+$, their Dirichlet and Neumann traces \textit{from the opposite side of the inteface} must vanish:

\begin{subequations}
\label{eq:VanishingTraces}
\noindent\begin{tabularx}{\textwidth}{@{}XX@{}}
  \begin{equation}
  \label{eq:VanishingTracesA}
\gamma^+\bs u = \bs 0\, , 
  \end{equation} &
  \begin{equation}
  \label{eq:VanishingTracesB}
\bs\sigma\bs\nu^+ = \bs 0\, ,
  \end{equation} \\[-.8cm]
    \begin{equation}
  \label{eq:VanishingTracesC}
\gamma^- v = 0\,, 
  \end{equation} &
  \begin{equation}
  \label{eq:VanishingTracesD}
\partial_\nu^-v = 0. 
  \end{equation}
\end{tabularx}
\end{subequations}
The equalities above are simply the mathematical statement of the physically intuitive fact that the elastic field does not extend outside the scatterer, while the acoustic field does not extend inside of it. This observation will be useful when recasting the problem in terms of boundary integral equations, which we will do in the following section.
%
\section{Reformulation as boundary integral equations}
%
\textbf{Direct integral representations.} We will use standard concepts and notation from BIEs as in the standard references \cite{HsWe2021,McLean2000}. A function defined in terms of the single layer potential $\mathrm S_f(s)$ and double layer potential $\mathrm D_f(s)$ as
\begin{equation}
\label{eq:IntegralRepresentationF}
v = \begin{cases} 0 & \text{ in }\; \Omega_- \\ \mathrm D_f(s)\gamma^+ v - \mathrm S_f(s)\partial_\nu^+ v & \text{ in }\; \Omega_+  \end{cases}
\end{equation}
will be the unique solution to the following exterior problem for the acoustic resolvent equation
\[
-\Delta v + (s/c)^2v = 0  \qquad  \text{ in } \; \Omega_+,
\]
whose exterior Dirichlet and Neumann traces are equal to $\gamma^+v$ and $\partial_\nu^+v$ respectively. Hence, we will propose an ansatz for $v$ in the form of \eqref{eq:IntegralRepresentationF} and use the conditions at the interface to determine the unknown traces $\gamma^+v$ and $\partial_\nu^+v$. To do that we will need to introduce the following boundary integral operators acting on densities $(\lambda_f,\phi_f)\in \rm H^{-1/2}(\Gamma)\times \rm H^{1/2}(\Gamma)$

\scalebox{1}{\(\displaystyle
\begin{array}{clccl}
\mathrm V_f(s)\lambda_f := \ave{\gamma \mathrm S_f(s)\lambda_f} = \gamma\mathrm S_f(s)\lambda_f & \quad\;\; \text{\small (Single layer)}\,, & \qquad \qquad & \mathrm K_f(s)\phi_f := \ave{\gamma\mathrm D_f(s)\phi_f}  & \quad\;\;\text{\small (Double layer)}\,, \\[1.5ex]
\mathrm K_f^\prime(s)\lambda_f := \ave{\partial_\nu\mathrm S_f(s)\lambda_f} &\quad\;\; \text{\small (Weakly singular)}\,, & \qquad \qquad & \mathrm W_f(s)\phi_f :=\ave{\partial_\nu\mathrm D_f(s)\phi_f} &\quad\;\;\text{\small (Hypersingular)}\,,
\end{array}
\)}

which satisfy the useful identities

\begin{subequations}
\label{eq:TraceIDs}
\noindent\begin{tabularx}{\textwidth}{@{}XX@{}}
  \begin{equation}
  \label{eq:TraceIDsA}
    \partial_\nu^{\pm}\mathrm S_f(s) = \mp \tfrac{1}{2} {\rm I} + \mathrm K_f^\prime(s)\,,  
  \end{equation} &
  \begin{equation}
  \label{eq:TraceIDsB}
\gamma^\pm \mathrm D_f(s) = \pm \tfrac{1}{2} {\rm I} + \mathrm K_f(s)\,. 
  \end{equation}
\end{tabularx}
\end{subequations}
Similar definitions can be made \cite{HsSaWe2022} for operators related to the interior elastic resolvent equation
\[
-\nabla\cdot\bs\sigma + (s\rho_\Sigma)^2\bs u = \bs 0  \qquad  \text{ in } \; \Omega_-,
\]
where the function $\boldsymbol u$ is extended by zero in $\Omega_-$ . We will denote the resulting layer potentials and operators with the same symbols as their acoustic counterparts, and will distinguish them by the use of bold font and the subscript $\Sigma$. The corresponding vector--valued densities $(\bs\phi_\Sigma,\bs\lambda_\Sigma)$ belong to the spaces $\mathbf {\bf H}^{-1/2}(\Gamma)\times {\bf H}^{1/2}(\Gamma)$.

We therefore propose an ansatz for $\bs u$ of the form
\begin{equation}
\label{eq:IntegralRepresentationS}
\bs u = \begin{cases} \bs{\mathrm S}_\Sigma(s)\bs\sigma\bs\nu^- - \bs{\mathrm D}_\Sigma(s)\gamma^-\bs u & \text{ in } \Omega_- \\ \boldsymbol 0 & \text{ in } \Omega_+ \end{cases}\,.
\end{equation}
Note that the extensions by zero of $v$ and $\boldsymbol u$ (into $\Omega_-$ and $\Omega_+$ respectively) guarantee that the conditions \eqref{eq:VanishingTraces} will be satisfied.

\noindent\textbf{A symmetric system of boundary integral equations.} The boundary integral operators and the identities \eqref{eq:TraceIDs} can be applied to the representations \eqref{eq:IntegralRepresentationF} and \eqref{eq:IntegralRepresentationS} to express the transmission conditions \eqref{eq:LDequationsC} and \eqref{eq:LDequationsD}, along with the vanishing trace conditions \eqref{eq:VanishingTracesA} and \eqref{eq:VanishingTracesD}, in terms of the unknown densities. This leads to a system of boundary integral equations for the densities $\partial_\nu^+ v, \gamma^+ v,\gamma^-\bs u$, and $\bs{\sigma\nu}^-$. To enforce the symmetry of the ensuing system, we introduce the auxiliary variable
\begin{equation}\label{eq:auxvar}
\eta = \gamma^+ v,
\end{equation}
and divide equations \eqref{eq:LDequationsC} and \eqref{eq:VanishingTracesA} by the fluid density $\rho_f$. All these ingredients together lead to:
\begin{equation}
\label{eq:BIE}
\left[\begin{matrix}
\mathrm W_f(s) & \phantom{,} & \left[\tfrac{1}{2}{\rm I} + \mathrm K_f^\prime(s)\right]  & \phantom{,} & 0 & \phantom{,} & 0 & \phantom{,} & 0 & \phantom{,}\\[1.1ex]
\left[\tfrac{1}{2} {\rm I} + \mathrm K_f(s)\right] & \phantom{,} & -{\rm V}_f(s) & \phantom{,} & 0 & \phantom{,} & 0 & \phantom{,} & -\rm I & \phantom{,}\\[1.1ex]
0 & \phantom{,} & 0 & \phantom{,} & \rho_f^{-1}\bs{\mathrm W}_{\Sigma}(s) & \phantom{,} & \rho_f^{-1}\left[\tfrac{1}{2}{\bf I} + \bs{\mathrm K}^\prime_{\Sigma}(s)\right] & \phantom{,} & s\bf N^\prime& \phantom{,}\\[1.1ex]
0 & \phantom{,} & 0 & \phantom{,} & \rho_f^{-1}\left[\tfrac{1}{2}\bf I + \bs{\mathrm K}_\Sigma(s)\right]  & \phantom{,} & -\rho_f^{-1}\bs{\mathrm V}_\Sigma(s) & \phantom{,} & 0 & \phantom{,}\\[1.1ex]
0 & \phantom{,} & -\rm I & \phantom{,} & s\bf N & \phantom{,} & 0 & \phantom{,} & 0 & \phantom{,}
\end{matrix}\right]
\begin{bmatrix}
\gamma^+ v \\[1.1ex]
\partial_\nu^+ v \\[1.1ex]
\gamma^-\bs u\\[1.1ex]
\bs{\sigma\nu}^- \\[1.1ex]
\eta
\end{bmatrix} = 
\begin{bmatrix}
0 \\[1.1ex]
0 \\[1.1ex]
-s{\bf N}^\prime\phi_0 \\[1.1ex]
\bs 0 \\[1.1ex]
-\lambda_0
\end{bmatrix}.
\end{equation}
Above, we have denoted the normal operator mapping acoustic to elastic trace spaces by
\begin{alignat*}{8}
\mathbf N: \mathbf H^{1/2}(\Gamma) \;&&\longrightarrow & \; \rm H^{1/2}(\Gamma)\,, \qquad \qquad& \mathbf N^\prime: \rm H^{1/2}(\Gamma) \; && \longrightarrow &\; \mathbf H^{1/2}(\Gamma)\,,\\
\boldsymbol\phi \;\;&& \longmapsto &\;\; \boldsymbol\phi\cdot\boldsymbol\nu & \phi \;\;&&\longmapsto &\;\; \phi\,\boldsymbol\nu\,.
\end{alignat*}
Regarding the rows of the system, the first and fourth correspond to the vanishing trace conditions \eqref{eq:VanishingTracesD} and \eqref{eq:VanishingTracesB} respectively; the third and fifth come from the transmission conditions \eqref{eq:LDequationsD} and \eqref{eq:LDequationsC} respectively; while the the second one is nothing but the definition of the mortar variable \eqref{eq:auxvar}.
Once the densities $\partial_\nu^+ v, \gamma^+ v,\bs{\sigma\nu}^-,\gamma^-\bs u$, and $\eta$ have been determined as solutions of the system above, the acoustic and elastic wavefields are then recovered from the integral representation formulas \eqref{eq:IntegralRepresentationF} and \eqref{eq:IntegralRepresentationS}, which guarantee that the functions indeed satisfy the PDEs. Alternatively, given a solution pair $(\bs u, v)$ to the PDE system \eqref{eq:LDequations}, the properties of the layer potentials guarantee the existence of  integral representations of the form \eqref{eq:IntegralRepresentationF}  and \eqref{eq:IntegralRepresentationS}. Hence, using the definitions of the boundary integral operators and the trace identities \eqref{eq:TraceIDs} it is easy to show that the densities $\gamma^+ v,\partial_\nu^+ v,\gamma^-\bs u,\bs{\sigma\nu}^-$, and $\eta$ solve the BIE system \eqref{eq:BIE}. We have thus proven the following
\begin{proposition}
\label{prop:equivalence}
The system of partial differential equations \eqref{eq:LDequations} and the system of boundary integral equations \eqref{eq:BIE} are equivalent in the sense that a solution of one completely determines a solution of the other one.
\end{proposition}

%
\section{Continuous well--posedness analysis}
%
Let us define the product space for the unknowns of the strong system \eqref{eq:BIE} by 
\[
\mathbb H : = {\rm H}^{1/2}(\Gamma)\times {\rm H}^{-1/2}(\Gamma)\times {\bf H}^{1/2}(\Gamma)\times {\bf H}^{-1/2}(\Gamma)\times {\rm H}^{1/2}(\Gamma)\,,
\]
and its dual space by $\mathbb H^\prime$. We will also consider the following \textit{closed subspaces} of the trace spaces
\[
{\bf Y}_h \subset {\bf H}^{1/2}(\Gamma), \;\; {\bf X}_h \subset {\bf H}^{-1/2}(\Gamma), \;\; {\rm Y}_h \subset {\rm H}^{1/2}(\Gamma) \;\; \text{ and } \;\; {\rm X}_h \subset {\rm H}^{-1/2}(\Gamma).
\]
The discrete analogues of $\mathbb H$ and $\mathbb H^\prime$ will be duly denoted $\mathbb H_h$ and $\mathbb H^\prime_h$ respectively. We will define the polar set of ${\bf Y}_h$ as
\[
{\bf Y}_h^\circ :=\{{\bs \lambda}\in {\bf X}_h : \langle{\bs \lambda}, {\bs \varphi}\rangle_\Gamma = 0 \;\; \forall{\bs \varphi}\in {\bf Y}_h\}.
\]
Analogous definitions can be made for the polar sets of all the remaining spaces above, and we will use the notation ${\bf X}^\circ_h,{\rm Y}^\circ_h$, and ${\rm X}^\circ_h$ to denote them. For densities $(\phi,\lambda,\bs\phi,\bs\lambda,\eta)\in \mathbb H$ and  $(\widetilde\phi,\widetilde\lambda,\widetilde{\bs\phi},\widetilde{\bs\lambda},\widetilde\eta)\in \mathbb H'$ we define the product duality pairing as
\[
\langle(\phi,\lambda,\bs\phi,\bs\lambda,\eta),(\widetilde\phi,\widetilde\lambda,\widetilde{\bs\phi},\widetilde{\bs\lambda},\widetilde\eta)\rangle_\Gamma : = \langle\phi,\widetilde\phi\rangle_\Gamma + \langle\lambda,\widetilde\lambda \rangle_\Gamma + \langle\bs\phi,\widetilde{\bs\phi}\rangle_\Gamma + \langle\bs\lambda,\widetilde{\bs\lambda} \rangle_\Gamma + \langle\eta,\widetilde{\eta} \rangle_\Gamma.
\]

Finally, if $\mathbb M$ is a $n\times n$ matrix, we will keep the matrix-vector product notation compact by defining
\[
\mathbb M(a_1,\ldots,a_n) : = (\mathbb M_{11}a_1 + \ldots + \mathbb M_{1n}a_n,\,\ldots\,,\mathbb M_{n1}a_1 + \ldots + \mathbb M_{nn}a_n) = {\small
\left(\mathbb M \begin{bmatrix} a_1 \\[-.65ex] \vdots \\[-.65ex] a_n\end{bmatrix} \right)^\top}.
\]
Note that, in order to keep the notation as light as possible, we will avoid the traditional use of ``$h$" as a subscript (or superscript) for the discretized variables. However, if we can prove that this discretized weak problem is well posed, then by taking $\mathbb H_h = \mathbb H$ we obtain the well posedness of the continuous problem as a consequence.

\noindent\textbf{A weak system of BIEs.} We are now in the position to start the analysis and we will start by posing the system of BIEs in weak form. Consider the matrix of operators 
\[
\mathbb M(s):= {\small \left[\begin{matrix}
\mathrm W_f(s) & \phantom{+} & \left[\tfrac{1}{2}{\rm I} + \mathrm K_f^\prime(s)\right]  & \phantom{+} & 0 & \phantom{+} & 0 & \phantom{+} & 0 & \phantom{+}\\[1.1ex]
\left[\tfrac{1}{2} {\rm I} + \mathrm K_f(s)\right] & \phantom{+} & -{\rm V}_f(s) & \phantom{+} & 0 & \phantom{+} & 0 & \phantom{+} & -\rm I & \phantom{+}\\[1.1ex]
0 & \phantom{+} & 0 & \phantom{+} & \rho_f^{-1}\bs{\mathrm W}_{\Sigma}(s) & \phantom{+} & \rho_f^{-1}\left[\tfrac{1}{2}{\bf I} + \bs{\mathrm K}^\prime_{\Sigma}(s)\right] & \phantom{+} & s\bf N^\prime& \phantom{+}\\[1.1ex]
0 & \phantom{+} & 0 & \phantom{+} & \rho_f^{-1}\left[\tfrac{1}{2}{\bf I} + {\bf K}_\Sigma(s)\right]  & \phantom{+} & -\rho_f^{-1}{\bf V}_\Sigma(s) & \phantom{+} & 0 & \phantom{+}\\[1.1ex]
0 & \phantom{+} & -\rm I & \phantom{+} & s\bf N & \phantom{+} & 0 & \phantom{+} & 0 & \phantom{+}
\end{matrix}\right]}
\]
defined implicitly by the left and right hand side of \eqref{eq:BIE}.

With the notation defined above, the weak formulation of the problem is:
\begin{alignat}{6}
\nonumber
&\text{Given data } \; (\phi_0,\lambda_0)\in {\rm H}^{1/2}(\Gamma)\times{\rm H}^{-1/2}(\Gamma), \text{ find } \;(\gamma^+ v,\partial_\nu^+ v,\gamma^-\bs u,\bs{\sigma\nu}^-,\eta)\in \mathbb H_h \; \text{ such that } \\[1ex]
\label{eq:weakBIE}
&\langle\,\mathbb M(s)(\,\gamma^+ v,\partial_\nu^+ v,\gamma^-\bs u,\bs{\sigma\nu}^-,\eta),(\phi,\lambda,\bs\phi,\bs\lambda,\widetilde\eta)\rangle_\Gamma = -\langle(0,0,s{\bf N}\phi_0,\bs 0,\lambda_0),(\phi,\lambda,\bs\phi,\bs\lambda,\widetilde\eta)\rangle_\Gamma \\[1ex]
\nonumber
&\text{for all } \; (\phi,\lambda,\bs\phi,\bs\lambda,\widetilde\eta)\in \mathbb H^\prime_h.
\end{alignat}
%

\noindent\textbf{Invertibility of the weak system.}  The acoustic and elastic Calder\'on projectors are defined respectively by
\begin{equation}\label{eq:twoprojectors}
\mathbf C_f(s) := \begin{bmatrix} 
\mathrm W_f(s) & \phantom{,} & \tfrac{1}{2}{\rm I} + \mathrm K_f^\prime(s) \\[1.1ex]
\tfrac{1}{2} {\rm I} + \mathrm K_f(s) & \phantom{,} & -{\rm V}_f(s)
\end{bmatrix}
\qquad \text{ and } \qquad
\mathbf C_\Sigma(s) := \begin{bmatrix}
\bs{\mathrm W}_{\Sigma}(s) & \phantom{,} & \tfrac{1}{2}{\bf I} + \bs{\mathrm K}^\prime_{\Sigma}(s)\\[1.1ex]
\tfrac{1}{2}{\bf I} + {\bf K}_\Sigma(s)  & \phantom{,} & -{\bf V}_\Sigma(s) 
\end{bmatrix}.
\end{equation}
The continuity and strong ellipticity of these operators on Lipschitz domains was established by Costabel, and Costabel and Stephan in the seminal articles \cite{Costabel:1988,CoSt:1990}. In the case of the elastic projector, uniqueness can only be guaranteed away from the eigenvalues of the Navier--Lam\'e operator---commonly known as \textit{Jones frequencies} \cite{Jones:1983}. Although it was widely believed---following a result by Harg\'e \cite{harge1990}---that these frequencies ``could not exist for arbitrary bodies", it has been recently established that such modes do indeed exist for Lipschitz domains \cite{DoNiSu2019,DoNiOv2022}.

\begin{theorem}
The system \eqref{eq:weakBIE} is uniquely solvable away from Jones frequencies.
\end{theorem}
\begin{proof}
Defining the matrices
$\;{\displaystyle \mathbb A(s) : = \begin{bmatrix} \mathbf C_f(s) & \phantom{,} &\bs 0 \\[1.1ex] \bs 0 & \phantom{,} & \rho_f^{-1}\mathbf C_\Sigma(s)\end{bmatrix}}$\; and\; ${\displaystyle
\mathbb B(s):= \begin{bmatrix}\; 0 & \phantom{,} & -\rm{I} & \phantom{,} & s\mathbf N & \phantom{,} & 0 \;\;\end{bmatrix}}$, we see that the system matrix $\mathbb M(s)$ has the saddle point structure
\[
\mathbb M(s) = \begin{bmatrix}
\; \mathbb A(s) & \phantom{+} & \mathbb B(s)^\top \\[1.1ex]
\mathbb B(s) & \phantom{+} & 0 
\end{bmatrix}.
\]
Since the diagonal blocks of $\mathbb M(s)$ are the acoustic and elastic Calder\'on projectors---which are bounded---and the normal vector operator $\bf N$ appearing in $\mathbb B(s)$ is bounded, it follows that $\mathbb M(s): \mathbb H_\Gamma \to \mathbb H_\Gamma^\prime$ is also bounded. The invertibility of the operator $\mathbb A(s)$ on Lipschitz domains, modulo Jones frequencies, was established in \cite{BaDjEs:2014} via an equivalent boundary value problem and, since $\mathbb B(s)$ is easily seen to be surjective, it then follows (see for instance \cite[Lemma 2.1 and Theorem 2.1]{Gatica2014}) that $\mathbb M(s)$ is invertible away from Jones frequencies.
\end{proof}
%
\textbf{Return to the time domain.} 
%
In the conforming setting, as laid out in the seminal article by Laliena and Sayas \cite{LaSa2009}, the standard process is to prove the solvability of the Laplace--domain system by recasting the boundary integral equations as a PDE problem with non--standard transmission conditions. This problem is then studied using a Lax-Milgram argument which yields stability bounds with explicit dependence on the Laplace parameter $s$ and its real part $\mathrm{Re}(s)$. Using the inversion theorem for the Laplace transform (\cite{DoSa2013} and \cite[Proposition 3.2.2]{Sayas2016errata,Sayas2016}), these explicit bounds can then be translated into time domain results that:
\begin{enumerate}
\item prescribe regularity conditions (in time) on the problem data that guarantee solvability of the time domain system, and
\item provide explicit polynomial--in--time bounds for the growth of the solutions---as opposed to the exponential--type bounds obtained through the use of Gronwall--style estimates. 
\end{enumerate}
Unfortunately, our use of \cite[Lemma 2.1 and Theorem 2.1]{Gatica2014} to prove the solvability of the saddle--point weak system \eqref{eq:weakBIE} does not yield an explicit expression for the stability constant. Thus, we can not exploit the machinery from \cite[Proposition 3.2.2]{Sayas2016errata,Sayas2016} to obtain time--domain results. However, due to the fact that the symmetric system \eqref{eq:BIE} that we analyze in this communication is equivalent to the one analyzed in \cite[Section 2.7]{HsSaSa2016} (in the sense of being able to recover the solution of either one of the systems from the solution to the other one), it is reasonable to expect time behavior analogous to \cite[Corollary 3.9]{HsSaSa2016} for the solution to the continuous problem, namely:

\begin{proposition}\label{cor:3.9}
Let the problem data $(\lambda_0,\phi_0)$ belong to the Sobolev spaces
\begin{align*}
\lambda_0 \in\,& \{ \xi\in \mathcal C^{2}(\mathbb R;{\rm H}^{-1/2}(\Gamma))\,:\, \xi \equiv 0 \text{ in } (-\infty,0), \xi^{(3)}\in L^1(\mathbb R;{\rm H}^{-1/2}(\Gamma))\},\\
\phi_0\in\,& \{ \varphi \in \mathcal C^{3}(\mathbb R;{\rm H}^{1/2}(\Gamma))\,:\, \varphi \equiv 0 \text{ in } (-\infty,0), \varphi^{(4)}\in L^1(\mathbb R;{\rm H}^{1/2}(\Gamma))\},
\end{align*}
where the notation $\phi^{(n)}$ is used to denote the $n$--th derivative of $\phi$ with respect to time. Then, the densities $\gamma^-\bs u, \boldsymbol\sigma\boldsymbol n^-,\gamma^+v$, and $\partial_\nu^+ v$ along with the solutions $(\mathbf{u},v)$ to the PDE system obtained from \eqref{eq:IntegralRepresentationF} and \eqref{eq:IntegralRepresentationS},  are continuous causal functions of time for all $t\ge 0$. Moreover, the post processed solutions satisfy the estimate
\[
\|(\boldsymbol{u},v)(t)\|_{1,\mathbb{R}^d\setminus\Gamma} 
\leq \frac{D_1t^2}{t+1}\max\{1,t^2\}\int_0^t\|\mathcal{P}_{3}(\lambda_0,\dot\phi_0)(\tau)\|_{-1/2,1/2,\Gamma}\;d\tau, 
\]
where $D_1$ depends only on the interface $\Gamma$ and 
\[
(\mathcal{P}_{k}f)(t) := \displaystyle\sum_{l=0}^{k} {k\choose l} f^{(l)}(t).
\]
\end{proposition}

In a similar vein, if we denote by $(\mathbf{e}^h,e^h):=(\boldsymbol{u}-\boldsymbol{u}^h,v-v^h)$ respectively the errors in the reconstructions of the elastic and acoustic fields $\boldsymbol u$ and $v$ resulting from a time semidiscretization of the system \eqref{eq:weakBIE}, we can expect a result simlar to \cite[Corollary 3.10]{HsSaSa2016}. Specifically:

\begin{proposition}\label{cor:3.10}
If the exact densities $\gamma^-\bs u$ and $\gamma^+ v$ are such that
\begin{align*}
\gamma^-\bs u\in\,& \{ \boldsymbol\varphi \in \mathcal C^{4}(\mathbb R;\mathbf H^{1/2}(\Gamma))\,:\, \boldsymbol\varphi \equiv 0 \text{ in } (-\infty,0), \boldsymbol\varphi^{(4)}\in \boldsymbol L^1(\mathbb R;\mathbf H^{1/2}(\Gamma))\},\\
 \gamma^+ v\in\,& \{ \varphi \in \mathcal C^{4}(\mathbb R;{\rm H}^{1/2}(\Gamma))\,:\, \varphi \equiv 0 \text{ in } (-\infty,0), \varphi^{(4)}\in L^1(\mathbb R; {\rm H}^{1/2}(\Gamma))\},
 \end{align*}
then $(\mathbf{e}^h,e^h)\in\mathcal{C}(\mathbb{R},\mathbf{H}^1(\mathbb{R}^d\setminus\Gamma)\times H^1(\mathbb{R}^d\setminus\Gamma))$  and for all $t\ge 0$ we have the bound
\[
\|(\mathbf{e}^h,e^h)(t)\|_{1,\mathbb{R}^d\setminus\Gamma} \leq \frac{D_2t^2}{t+1}\max\{1,t^2\}\int_0^t\|\mathcal{P}_{4}(\gamma^-\bs u-\boldsymbol\Pi_h\gamma^-\bs u,\gamma^+ v -\Pi_h\gamma^+ v)(\tau)\|_{1/2,\Gamma}\;d\tau,
\]
where $\boldsymbol\Pi_h$ and $\Pi_h$ are the best approximation operators in $\mathbf Y_h$ and $Y_h$, and $D_2$ depends only on $\Gamma$.
\end{proposition}
%
\section{Concluding remarks}
%
\textbf{Ease of implementation.}
An attractive feature of the proposed formulation is that, as opposed to the two--unknown formulation studied in \cite{HsSaSa2016,HsSaWe2015,Sanchez2016}, it does not require the discretization of combined acoustic/elastic operators. This feature, together with the use of convolution quadrature for time evolution, implies that any preexisting boundary integral codes for the stationary acoustic and elastic problems can be used together with a CQ implementation for time--domain wave/structure computations. 

Schematically, the discretization matrix takes the form
\[
\begin{pmatrix} \begin{bmatrix}\mathbf C_f^h(s) & \phantom{0} &\boldsymbol 0 \\ & & \\ \boldsymbol 0 & \phantom{0} &\mathbf C_\Sigma^h(s) \end{bmatrix} & \begin{bmatrix}  \phantom{0} \\ \;\mathbf R^\top \\ \phantom{0} \end{bmatrix} \\[3.5ex]  \begin{bmatrix}\phantom{0} \qquad\;\, \mathbf R \qquad\;\, \phantom{0} \end{bmatrix} & \boldsymbol 0 \end{pmatrix}\,, \qquad \text{ with } \qquad \mathbf{R}:= [\; 0 \quad -\mathrm{I}^h \quad s\mathbf N^h \quad 0 \quad 0 \;] \,,
\]
where $\mathbf C_f^h(s)$ and $\mathbf C_\Sigma^h(s)$ are, respectively the discretizations of the acoustic and elastic Calder\'on projectors appearing in  \eqref{eq:twoprojectors}, each of which can be obtained from independent preexisting solvers. The codes communicate only through the matrix $\mathbf{R}$ that amounts to the mass matrix in the acoustic discretization (for the identity), and a rectangular matrix $\mathbf N^h$ with the $x$ and $y$ components of the exterior unit normal vector on each element of the discretization of the acoustic/elastic interface.

\textbf{Numerical considerations.}
We stress that the central point of this article is to introduce and prove the well--posedness of the formulation leading to the system \eqref{eq:BIE}, rather than to advocate for one particular discretization scheme over another one. On the contrary, the main advantage of the proposed formulation is that it is amenable for discretization with any off--the--shelf acoustic and elastic codes that the reader is familiar and comfortable with---obviating the need for specialized treatment of the coupling terms. The following numerical example is included for the sole purpose of showcasing the fact that the formulation can indeed be implemented. For the sake of completennes, a very brief description of our particular choice of numerical method is included below.

For the spatial discretization, we use {\tt deltaBEM}, a Nystr\"om--flavored discretization for the operators of the acoustic \cite{DoLuSa2014a,DoLuSa2014b} and elastic \cite{DoSaSa2015} Calder\'on projector. The method makes use of look--around quadrature formulas for the approximation of boundary integrals and collocates the discrete equations on symmetrically staggered grids to achieve third--order spatial accuracy. As for the elastic hypersingular operator, the code discretizes a regularization due to Frangi and Novati \cite{Frangi1998} that results in an integro--differential operator. The reader is directed to the aforementioned works for further details and to \cite{deltaBEM} for a freely--available implementation of {\tt deltaBEM} that also includes Convolution Quadrature capabilities.

For the discrete time evolution, we use second--order backwards--differentiation convolution quadrature, typically referred to as BDF2-CQ. Convolution quadrature algorithms combine time--domain problem data, Laplace--domain convolution kernels (typically tamer than their temporal counterparts), and weights associated to a time--stepping method to approximate the convolution integrals appearing in the time--domain layer potentials. A detailed discussion of these methods falls way beyond the scope of this communication; the reader is referred to \cite{HaSa2016} for a detailed and accessible exposition of the theoretical and implementation aspects of CQ. For the purposes of this communication, we shall consider CQ as a blackbox time--stepping mechanism that serves two purposes:
\begin{enumerate}
\item Given \textit{time--domain} boundary data, as appearing in the right hand side of equations \eqref{eq:TDequationsC} and \eqref{eq:TDequationsD}, produces the Laplace--domain boundary data appearing on the right hand side of the system \eqref{eq:BIE}.

\item Given \textit{Laplace--domain} solutions to the BIE system \eqref{eq:BIE}, produces \textit{time--domain} solutions to the system \eqref{eq:TDequations}.
\end{enumerate}

%
\textbf{An acoustic/elastic simulation.}
The following example is included for demonstrative purposes. We consider a kite--shaped elastic body with elastic coefficients 
\[
\rho_\Sigma = 2\times 10^{7} {\rm kg/m^2}\,,\quad \lambda = 5\times 10^{6} {\rm Pa}, \quad \text{ and } \quad \mu = 4\times 10^{6} {\rm Pa},
\]
sorrounded by an irrotational fluid of density $\rho_f=1 {\rm kg/m^2}$ and wave speed $c=1\,{\rm m./s.}$. The elastic parameters yield transversal and logitudinal wave speeds of $c_T \approx 0.45\,\,{\rm m/s}$ and $c_L\approx 0.81\, \,{\rm m/s}$ respectively. These parameter values are not meant to replicate any particular media; they were chosen to obtain comparable wave speeds in both acoustic and elastic media. This allows both acoustic and elastic waves to be visualized on the same time scale in the numerical experiment. The solid is initially at rest, and is hit by an acoustic wave pulse of the form
\[
v^{inc} = 2W(t)\sin^3(2t)\,,
\]
propagating in the direction $\mathbf d = (1,1)$. The windowing function $W(t)$ is a $\mathcal C^6$ approximation to the characteristic function of the interval $[0,\pi/2]$. The simulation was computed using 600 discretization points for the boundary of the kite, and a second--order backward Euler--based Convolution Quadrature formula (the reader is referred to \cite{HaSa2016} for theoretical and implementation details pertaining CQ in general and BDF2--CQ in particular) with a time step of size $\Delta t = 2.5\times 10^{-3}$. Figure \ref{fig:Simulation} features snapshots of the interaction at different times. For the elastic field (inside the kite--shaped scatterer) the magnitude of the elastic displacement is plotted on a \textit{heatmap} color scheme, where black represents no displacement and gradually brighter shades of orange and yellow denote increasing values of the magnitude. A similar gray--scale color scheme is used for the acoustic pressure. The simulation accurately predicts that the incident wave excites an elastic perturbation that propagates at a slower speed than that of the exterior acoustic wave.

\begin{figure}
\begin{tabular}{ccc}
 {\large $t=0$} & {\large $t=1.2$} & {\large $t=2.4$}\\[2ex] 
\includegraphics[width = 0.3\linewidth]{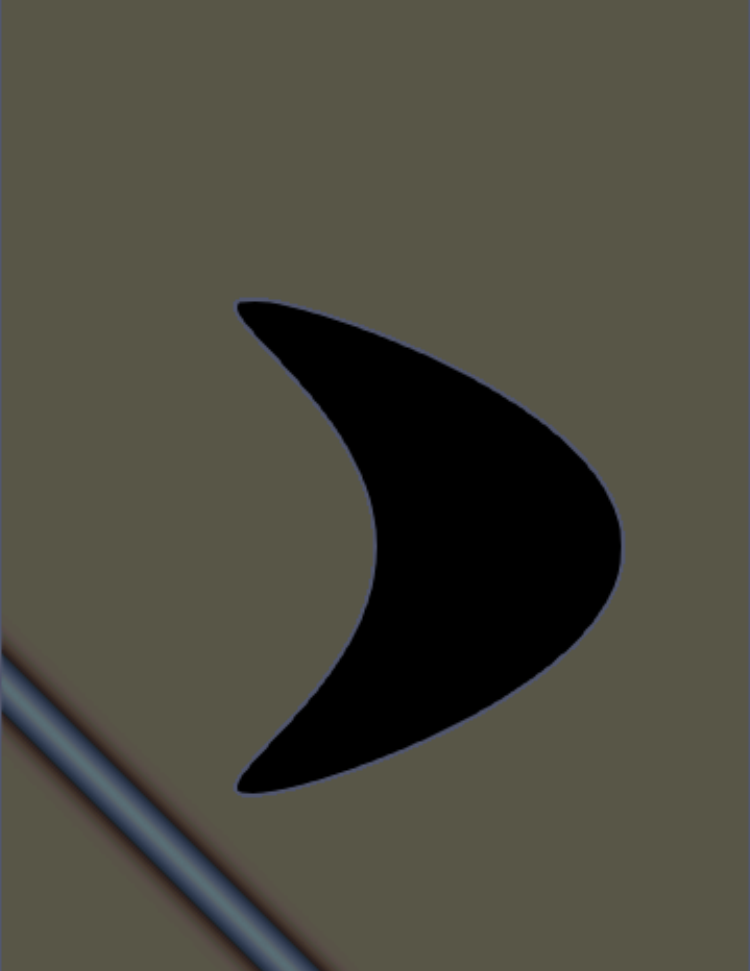} &
\includegraphics[width = 0.3\linewidth]{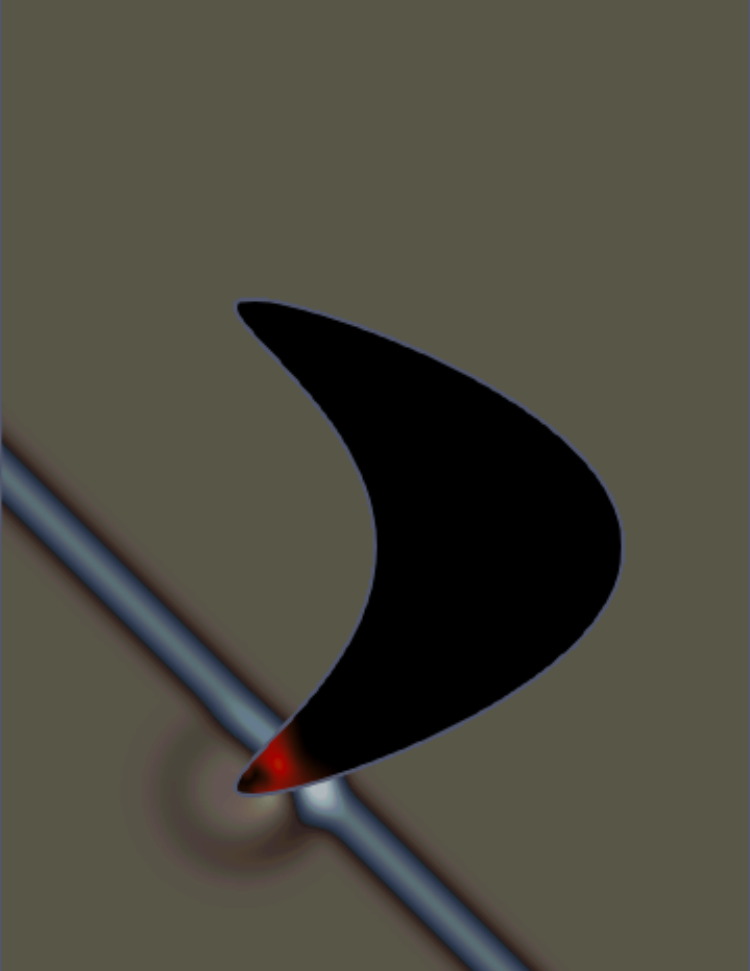} &
\includegraphics[width = 0.3\linewidth]{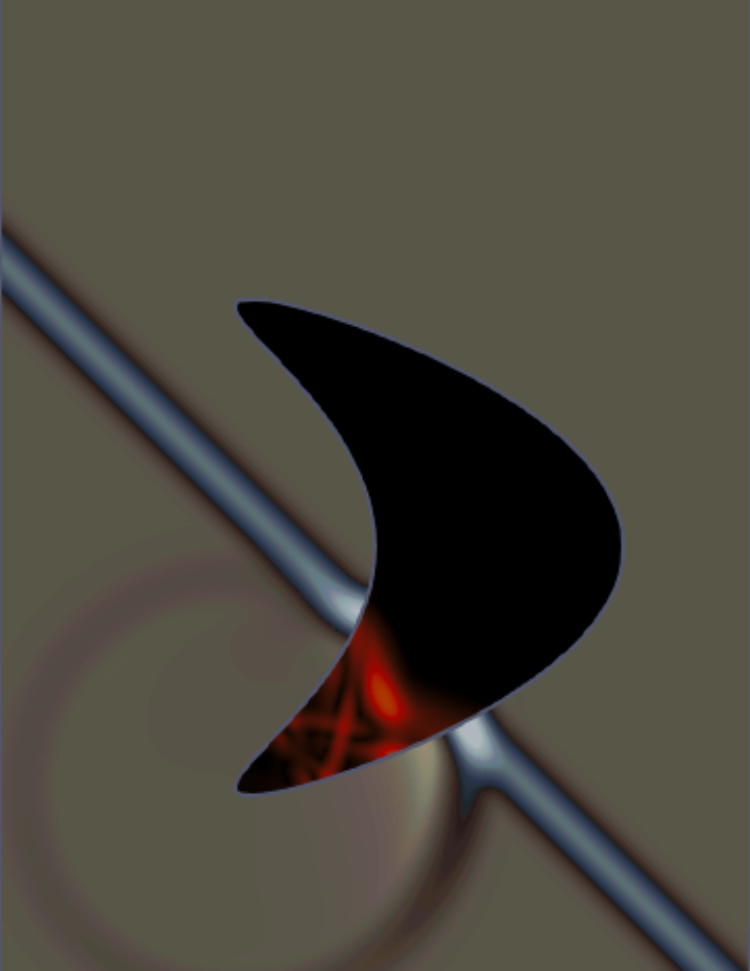} \\[2ex]
{\large $t=3.6$} & {\large $t=4.8$} & {\large $t=6$}\\[2ex]
\includegraphics[width = 0.3\linewidth]{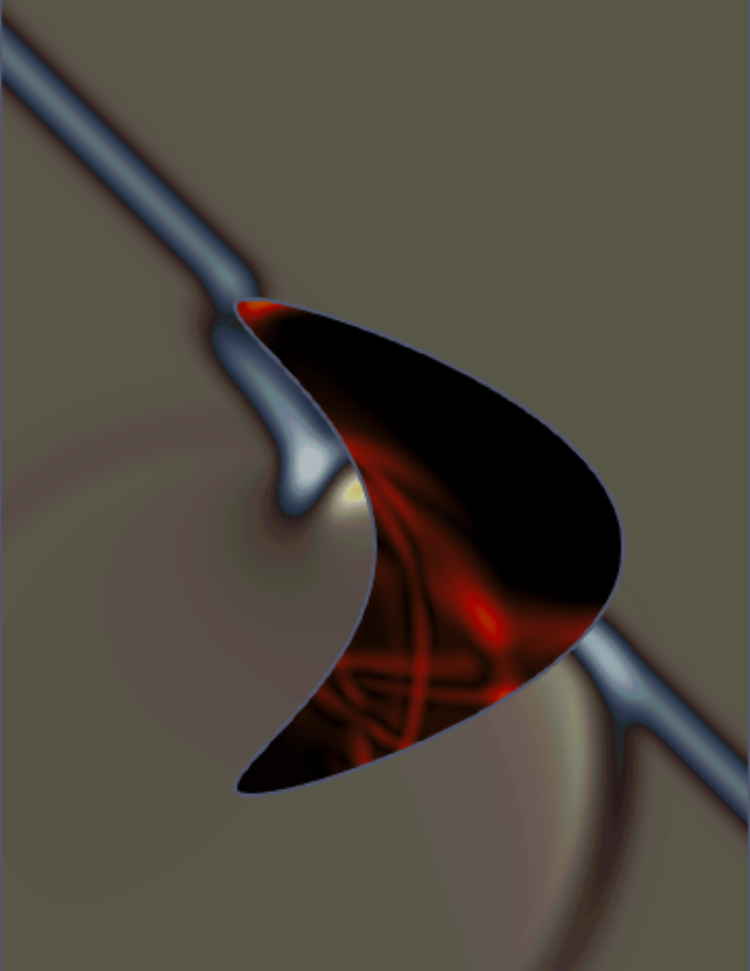} &
\includegraphics[width = 0.3\linewidth]{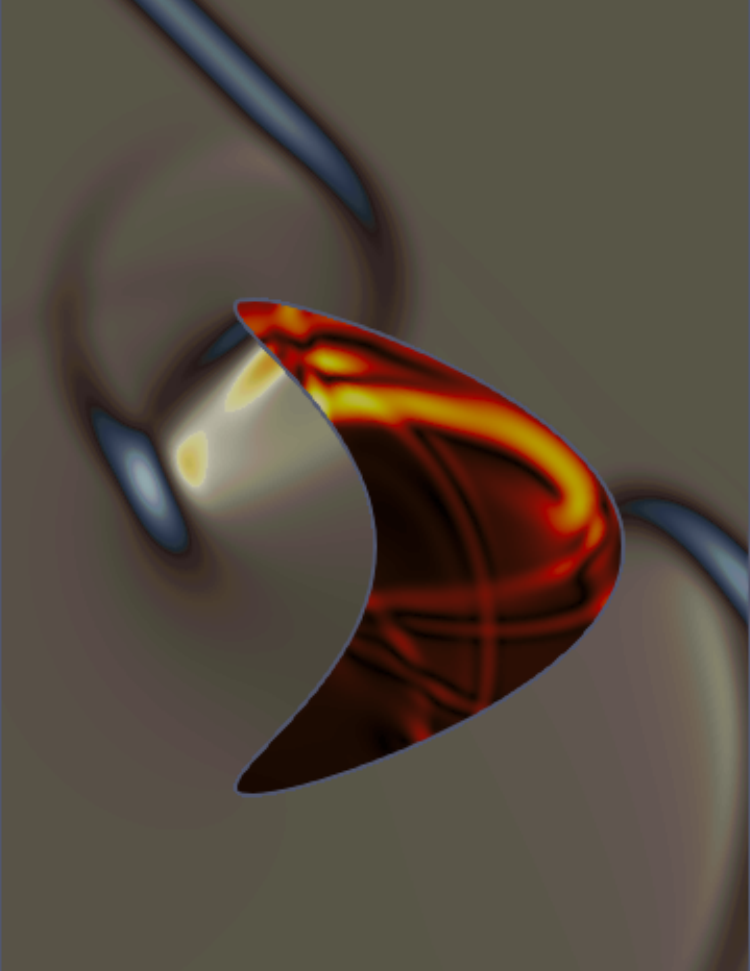} &
\includegraphics[width = 0.3\linewidth]{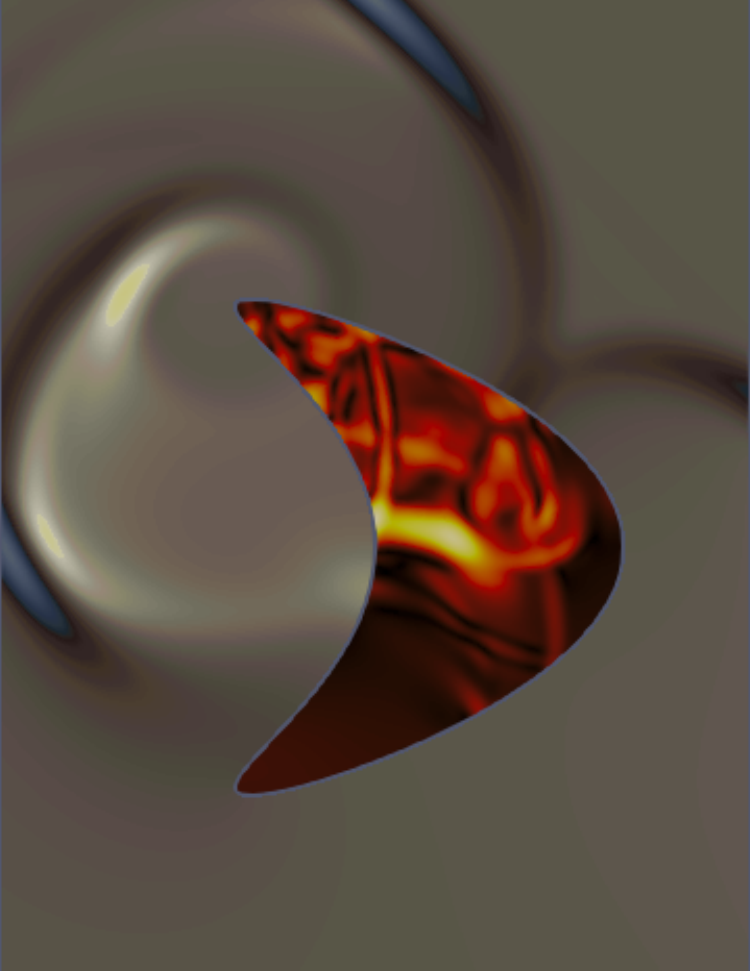}
\end{tabular}
\caption{An acoustic pulse incides on an elastic object. The heatmap in the elastic domain indicates the magnitude of the elastic displacement $\boldsymbol u$.}\label{fig:Simulation}
\end{figure}
%
\section*{Acknowledgements}
%
Tonatiuh S\'anchez-Vizuet has been partially funded by the U. S. National Science Foundation through the grant NSF-DMS-2137305.

%
\clearpage
\bibliographystyle{abbrv}
\bibliography{references}

\end{document}